\documentclass[12 pt, leqno]{article}
\usepackage{amssymb,amsmath,amsfonts,amsthm,bbm}
\usepackage{accents,oldgerm,picinpar,epsfig,graphics,graphicx,mathrsfs}
\usepackage[none]{hyphenat}
\usepackage{enumerate, fancyhdr, lastpage}

\DeclareMathAlphabet{\chan}{T1}{pzc}{mb}{it}

\newcommand{\op}{\operatorname}

\newcommand{\mc}{\mathcal}

\newcommand{\abs}[1]{\left\vert#1\right\vert}

\newcommand{\set}[1]{\left\{#1\right\}}
\newcommand{\ignore}[1]{}

\newcommand{\defined}[1]{\textbf{#1}}

\newcommand{\pr}{\mathrm{pr}}
\newcommand{\w}{\omega}

\newcommand{\U}{\mathcal U}

\newcommand{\uhr}{\upharpoonright}

\newcommand{\vc}{\v{C}ech}

\newtheorem{prob}{Problem}

\newsavebox{\Theoremtype}
\newsavebox{\Theoremlabel}

\newtheoremstyle{break}
{\topsep}
{\topsep}
{\it}
{}
{}
{}
{ }
{\thmname{\textbf{#1}}\thmnumber{ \textbf{#2.}}}

\newtheoremstyle{ref}
{\topsep}	
{\topsep}	
{\it}
{}
{}
{}
{ }
{\thmname{{\bfseries#1}}\thmnumber{ \bfseries#2\thmnote{~\rm #3}\bfseries .}}
\theoremstyle{ref}
\newtheorem{thm}{Theorem}

\newtheorem{lem}[thm]{Lemma}
\newtheorem{claim}[thm]{Claim}
\newtheorem{prop}[thm]{Proposition}

\newtheorem{cor}[thm]{Corollary}
\newtheorem{ex}{Example}

\newtheoremstyle{nnref}
{\topsep}	
{\topsep}	
{\it}
{}
{}
{}
{ }
{\thmname{\textbf{#1}\thmnote{\textrm{ #3}}\textbf{.}}}

\theoremstyle{nnref}
\newtheorem{defn}{Definition}

\begin{document}
\sloppy
\title{Productively Lindel\"of and Indestructibly Lindel\"of Spaces}
\author{Haosui Duanmu, Franklin D. Tall\makebox[0cm][l]{$^1$} \ and Lyubomyr Zdomskyy\makebox[0cm][l]{$^2$}}

\footnotetext[1]{Research supported by NSERC grant A-7354.  Some of this work was undertaken at the Kurt G\"odel Research Center at the University of Vienna in July, 2012. The second author thanks that institution and especially the third author for their hospitality.\vspace*{1pt}}
\footnotetext[2]{The third author would like to thank FWF grant M1244-N13 for support for this research.}
\date{\today}
\maketitle

\begin{abstract}

\end{abstract}

\renewcommand{\thefootnote}{} \footnote
{\parbox[1.8em]{\linewidth}{$(2010)$ AMS Mathematics Subject
    Classification. Primary 54A25, 54A35, 54D20, 54B10, 03E35.}\vspace*{3pt}} \renewcommand{\thefootnote}{}
\footnote {\parbox[1.8em]{\linewidth}{Key words: productively Lindel\"of, indestructibly Lindel\"of, Rothberger, $\aleph_1$-Hurewicz, $\aleph_1$-Borel Conjecture, projectively $\sigma$-compact.}}
There has recently been considerable interest in \emph{productively Lindel\"of spaces}, i.e. spaces such that their product with every Lindel\"of space is Lindel\"of.  See e.g. \cite{AT}, \cite{TT}, \cite{AAJT}, \cite{T}, and work in progress by Miller, Tsaban, and Zdomskyy, Repovs and Zdomskyy, and by Brendle and Raghavan.  Here we make several related remarks about such spaces.  \textit{Indestructible Lindel\"of spaces}, i.e. spaces that remain Lindel\"of in every countably closed forcing extension, were introduced in \cite{aT1}.  Their connection with topological games and selection principles was explored in \cite{ST}.  We find further connections here.

\section{A sufficient condition for a space not to be productively Lindel\"of}
In \cite{AAJT}, a set of four conditions was given for a regular Lindel\"of space $X$ of countable type to be not productively Lindel\"of.  The conditions and proof were unnecessarily complicated because the authors wanted to produce a \textit{regular} $Z$ such that $X \times Z$ was not Lindel\"of.  This extra effort was not necessary because we can prove the following result.

\begin{lem}\label{Lem1}
Let $X$ be a Lindel\"of space.  If there is a Lindel\"of space $Z$ such that $X \times Z$ is not Lindel\"of, then there is a such a Lindel\"of $Z'$, which furthermore is $0$-dimensional $T_1$ and hence regular.
\end{lem}
\begin{proof}
Let $\set{U_\alpha \times V_\alpha : \alpha < \kappa}$ be an open cover of $X \times Z$ which does not have a countable subcover.  Consider the following set-valued maps:
\[\Phi_X : X \to 2^{\kappa}, \Phi_X(x) = \set{A \subseteq \kappa : \set{\alpha : x \in U_\alpha} \subseteq A},\]
\[\Phi_Z : Z \to 2^{\kappa}, \Phi_Z(z) = \set{C \subseteq \kappa : \set{\alpha : z \in V_\alpha} \subseteq C},\]
\[\Phi_{X,Z} : X \times Z \to 2^{\kappa}, \Phi_{X, Z}(x,z) = \set{B \subseteq \kappa : \set{\alpha : \langle x, z \rangle \in U_\alpha \times V_\alpha} \subseteq B}\]

By Lemma 2 of \cite{Z} each of these maps is compact-valued and upper semicontinuous.  Lindel\"ofness is preserved by compact-valued upper semicontinuous maps, so $T = \Phi_Z(Z) \subseteq 2^{\kappa}$ is Lindel\"of.  To show $X \times T$ is not Lindel\"of, it suffices to show $X' \times T$ is not Lindel\"of, where $X' = \Phi_X(X)$.  Consider the map
\[\Phi : X' \times T \to 2^{\kappa}, \Phi(A,C) = A \cap C.\]
Notice that $A \cap C \neq \emptyset$ for any $A \in X'$ and $C \in T$.  Indeed, find $\langle x, z \rangle \in X \times Z$ such that $\set{\alpha : z \in V_\alpha} \subseteq C$ and $\set{\alpha : x \in U_\alpha} \subseteq A$.  Let $\alpha$ be such that $\langle x, z \rangle \in U_\alpha \times V_\alpha$.  Then $\alpha \in A \cap C$.  It follows from the above that $\mc{W} = \set{W_\alpha : \alpha < \kappa}$, where $W_\alpha = \set{D \subseteq \kappa : \alpha \in D}$ is an open cover of $X' \times T$.  However $\mc{W}$ has no countable subcover, since if $\set{W_\alpha : \alpha \in I}$ covers $X' \times T$, $\set{U_\alpha \times V_\alpha : \alpha \in I}$ is a cover of $X \times Z$.  Thus $X' \times T$ is not Lindel\"of and hence neither is $X \times T$.
\end{proof}

\begin{defn}
$L(X)$, the \defined{Lindel\"of number} of $X$, is the least cardinal $\lambda$ such that every open cover of $X$ has a subcover of size $\leq \lambda$.  The \defined{type} of $X$, $T(X)$, is the least cardinal $\kappa$ such that for each compact $L \subseteq X$ there is a compact $K$ including $L$ such that there is a base of size $\leq \kappa$ for the open sets including $K$. (For $T_{3\frac{1}{2}}$ $X$, this is equivalent to $L(\beta X - X) \leq \kappa$.)  The \defined{weight} of $X$, $w(X)$, is  the least cardinal of a base.
\end{defn}
Note that $T(X) \leq w(X)$.  From Lemma \ref{Lem1}, we obtain the following simplified version of the main theorem of \cite{AAJT}.  The proof is also a simplified version of that in \cite{AAJT}, so will be omitted.

\begin{thm}\label{Thm2}Let $\langle X, \mathcal{T} \rangle$ be a Lindel\"of space of countable type.  Suppose there is a $Y \subseteq X$ and a topology $\rho$ on $Y$ such that
\begin{enumerate}[i)]
\item{
$\mathcal{T}|Y \subseteq \rho$,
}
\item{
$\langle Y, \rho \rangle$ is not Lindel\"of,
}
\item{
any $K \subseteq X$ that is $\mathcal{T}$-compact is such that $K \cap Y$ is $\rho$-Lindel\"of.
}
\end{enumerate}
Then $X$ is not productively Lindel\"of.  Indeed there is a regular Lindel\"of $Z$ such that $X \times Z$ is not Lindel\"of.
\end{thm}

The authors of \cite{AAJT} observe the following corollary.

\begin{cor}\label{Cor3}
Let $X$ be a Lindel\"of regular space of countable type.  If there is an uncountable $Y \subseteq X$ such that for each compact subset $K$ of $X$, $K \cap Y$ is countable, then $Y$ is not productively Lindel\"of.
\end{cor}

\section{L-productive spaces}
\begin{defn}
A space $X$ is \defined{$\leq \kappa$-L-productive} if $L(X \times Y) \leq L(Y)$ whenever $\aleph_0 \leq L(Y) \leq \kappa$.  A space $X$ is \defined{L-productive} if $L(X \times Y) \leq L(Y)$ for all $Y$.  A space $X$ is \defined{powerfully Lindel\"of} if $X^\omega$ is Lindel\"of.
\end{defn}

Despite much effort, the following problem of E. A. Michael remains unsolved.

\begin{prob}
If $X$ is productively Lindel\"of, is $X$ powerfully Lindel\"of?
\end{prob}

The best result so far is:

\begin{lem}[\cite{A}]\label{Lem4Prea}
The Continuum Hypothesis (CH) implies that if $X$ is productively Lindel\"of and regular and $w(X) \leq \aleph_1$, then $X^\omega$ is Lindel\"of.
\end{lem}

Note that L-productive spaces are productively Lindel\"of.  Thus a more modest problem is:

\begin{prob}
Is every L-productive space powerfully Lindel\"of?
\end{prob}

We shall make some small progress toward solving this problem.  Since we occasionally will deal with spaces that are not necessarily Lindel\"of, it is convenient to assume from now on that all spaces are Tychonoff.

\begin{defn}
$Y \subseteq X$ is \defined{sKinny} if $\abs{Y \cap K}< \abs{Y}$ for every compact $K \subseteq X$.  A collection $\mathcal{G}$ of subsets of $X$ is a \defined{$k$-cover} if every compact subset of $X$ is included in a member of $\mathcal{G}$.  $A(X)$, the \defined{Alster degree} of $X$, is the least cardinal $\kappa$ such that every $k$-cover of $X$ by $G_\delta$'s has a subcover of size $\leq \kappa$.  If $A(X) \leq \aleph_0$, we say $X$ is \defined{Alster}.
\end{defn}

\begin{defn}
A space $X$ is \defined{$\aleph_1$-L-productive} if $L(X \times Y) \leq \aleph_1$ whenever $L(Y) \leq \aleph_1$.
\end{defn}
\noindent Note that this does not imply productively Lindel\"of.

\begin{thm}[{\cite{A}}]\label{ThmA}
Alster spaces are powerfully Lindel\"of.
\end{thm}

\begin{lem}\label{Lem4a}
$\aleph_2^{\aleph_0} = \aleph_2$ implies if $w(X) \leq \aleph_2$ and $A(X) = \aleph_2$, then $X$ has a sKinny subspace of size $\aleph_2$.
\end{lem}

\begin{proof}
Let $\mathcal{G}$ be a $k$-cover of $X$ by $G_\delta$'s which has no subcover of size $\leq \aleph_1$.  By hypothesis we may assume that $\mathcal{G} = \set{G_\alpha}_{\alpha < \omega_2}$.  Pick $x_\alpha \in X - \left(\bigcup_{\beta < \alpha}G_\beta \cup \set{x_\beta : \beta < \alpha}\right)$.  This defines $A = \set{x_\alpha : \alpha < \omega_2}$, for if the construction stopped at $\gamma < \omega_2$, by taking $\set{G_\beta : \beta < \gamma}$ together with a member of $\mathcal{G}$ containing $x_\beta$, for each $\beta < \gamma$, we would obtain a subcover of $\mathcal{G}$ of size $\leq \aleph_1$, contradiction.  $A$ is sKinny since $\mathcal{G}$ is a $k$-cover.
\end{proof}

\begin{thm}\label{Thm4a}
If $X$ is Lindel\"of, and if $T(X) \leq \aleph_1$ and $X$ has a sKinny subspace of size $\aleph_2$, then $X$ is not $\aleph_1$-L-productive.
\end{thm}
\begin{proof}
This is accomplished by a straightforward generalization of Theorem \ref{Thm2} and Corollary \ref{Cor3}.  See \cite{AAJT} for their proofs.
\end{proof}

\begin{thm}\label{Thm7a}
If CH and $2^{\aleph_1} = \aleph_2$, then every Lindel\"of\\$\leq \aleph_1$-L-productive space with $T \leq \aleph_1$ and $w \leq \aleph_2$ is powerfully Lindel\"of.
\end{thm}
\begin{proof}
If $A(X) = \aleph_0$, then $X^\omega$ is Lindel\"of by Theorem \ref{ThmA}.  If $A(X) = \aleph_1$, then $L(X^{\omega}) \leq \aleph_1$ by repeating the proof of Theorem \ref{ThmA} in \cite{A}.  But we have:
\begin{lem}[\cite{BT}]\label{Lem8a}
CH implies that if $X$ is productively Lindel\"of and $L(X^\omega) \leq \aleph_1$, then $X$ is powerfully Lindel\"of.
\end{lem}
Finally, if $A(X) = \aleph_2$, then $X$ has a sKinny subspace of size $\aleph_2$ by Lemma \ref{Lem4a}.  Then by Theorem \ref{Thm4a}, $X$ is not $\aleph_1$-L-productive, a contradiction.
\end{proof}

Unfortunately, we do not know how to generalize Theorem \ref{Thm7a} to higher weights, even for spaces of countable type, because of the dependence of the proof of Lemma \ref{Lem8a} on Lemma \ref{Lem4Prea}.  However, we do have a variation of Theorem \ref{Thm7a}:

\begin{thm}
Suppose CH and $2^{\aleph_1} = \aleph_2$.  Then every Lindel\"of $\leq \aleph_1$-L-productive space with $T \leq \aleph_1$ and size $\leq \aleph_2$ is powerfully Lindel\"of.
\end{thm}
\begin{proof}
Take a countably closed elementary submodel $M$ of $H_{\theta}$ of size $\aleph_2$, $\theta$ sufficiently large and regular, with $X$ and its topology in $M$.  Without loss of generality, assume $X \subseteq M$.  $X_M$ \cite{JT} is the topology on $X \cap M$ generated by $\set{U \cap M : U \in M, U \text{ open in } X}$.  In the special case of $X \subseteq M$, $X_M$ is just a weaker topology on $X$.
\begin{lem}[\cite{JT}]
For $X$ of countable type, $X_M$ is a perfect image of a subspace of $X$.  Furthermore, each $x \in X_M$ is a member of its pre-image.
\end{lem}

It follows that if $X$ is of countable type and $X \subseteq M$, then $X_M$ is a perfect image of $X$.

\begin{lem}
Perfect maps preserve countable type.
\end{lem}

This is probably due to Arhangel'ski\u{i}; it is quoted without attribution in \cite{B}.  It follows that a Lindel\"of $\leq \aleph_1$-L-productive space $X$ of countable type and size $\leq \aleph_2$ will map onto a Lindel\"of $\leq \aleph_1$-L-productive $X_M$ of countable type and weight $\leq \aleph_2$, which is then powerfully Lindel\"of.  As in \cite{TT}, we argue that $(X^\omega)_M = (X_M)^\omega$.  If there were an open cover $\mc{U}$ of $X^\omega$ without a countable subcover, there would be one in $M$.  Then $\set{U \cap M : U \in \mc{U} \cap M}$ would cover $(X^\omega)_M$.  Take $\mc{U}' = \set{U_n : n < \omega} \subseteq \mc{U} \cap M$, a countable subcover.  $M$ is countably closed so $\mc{U}' \in M$.  Then $M \models \mc{U}' \text{ covers } X^\omega$, so $\mc{U}'$ does cover $X^\omega$.
\end{proof}


\section{Rothberger and indestructible spaces and the $\aleph_1$-Borel Conjecture}
\begin{defn}
A space is \defined{Rothberger} if for each sequence $\set{\mc{U}_n}_{n < \omega}$ of open covers of $X$, there are $U_n \in \mc{U}_n$ such that $\set{U_n : n < \omega}$ is an open cover.  Equivalently \cite{P}, if Player ONE does not have a winning strategy in the game $\mathbf{G}_1^{\omega}(\mc{O}, \mc{O})$ in which in inning $n$, ONE picks an open cover and Player TWO picks an element of it, with ONE winning if the selections do not form an open cover.  A space is \defined{indestructible} if it generates a Lindel\"of topology in any countably closed forcing extension.  Equivalently \cite{ST} if ONE does not have a winning strategy in the $\omega_1$-length game $\mathbf{G}_1^{\omega_1}(\mc{O}, \mc{O})$ defined analogously to $\mathbf{G}_1^{\omega}(\mc{O}, \mc{O})$.
\end{defn}
\begin{defn}
A space $X$ is \defined{projectively countable} if whenever $f : X \to Y$, $Y$ separable metrizable (equivalently, $Y = \mathbb{R}$ or $Y = [0, 1]^{\omega}$, or etc.), $f(X)$ is countable.  \defined{Projectively $\sigma$-compact} is defined similarly.  $X$ is \defined{projectively $\aleph_1$} if whenever $f : X \to [0,1]^{\omega_1}$, then $\abs{f(X)} \leq \aleph_1$.
\end{defn}
Projectively countable Lindel\"of spaces are Rothberger \cite {K}, \cite{BCM}, \cite{T}; in fact
\begin{prop}[\cite{M}]\label{Prop4}
Borel's Conjecture is equivalent to the assertion that a space is Rothberger if and only if it is projectively countable.
\end{prop}

Surprisingly, productive Lindel\"ofness can substitute for Borel's Conjecture:

\begin{thm}\label{Thm5}
Suppose $X$ is a productively Lindel\"of Rothberger space.  Then $X$ is projectively countable.
\end{thm}
\begin{proof}
By Corollary \ref{Cor3}, it suffices to show that if $f : X \to \mathbb{R}$, then compact subspaces of $f(X)$ are countable.  If such a subspace were uncountable, it would include a perfect subset and hence a copy of the Cantor set.  But then a closed, hence Rothberger subset of $X$ would map onto the Cantor set, which is not Rothberger.
\end{proof}

Since indestructibility is the game version of Rothberger up one cardinal, it is reasonable to see whether Proposition \ref{Prop4} and Theorem \ref{Thm5} have generalizations to indestructibility.  One difficulty we should first dispose of is the question of whether indestructibility is the right generalization of Rothberger for this context, or whether it is more appropriate to consider the selection principle variation:

\begin{defn}
A space is \defined{$\omega_1$-Rothberger} if whenever $\set{\mc{U}_\alpha}_{\alpha < \omega_1}$ are open covers, there is a selection $U_\alpha \in \mc{U}_\alpha$, $\alpha < \omega_1$, such that $\bigcup\set{U_\alpha : \alpha < \omega_1}$ is a cover.
\end{defn}

In \cite{ST}, Scheepers and Tall ask whether $\omega_1$-Rothberger is the same as indestructible.  The latter easily implies the former but Dias and Tall \cite{D} exhibit a destructible Lindel\"of space which, under CH, is $\omega_1$-Rothberger.

\begin{ex}\label{DExample}
The lexicographic order topology on $2^{\omega_1}$ is a compact destructible (see \cite{D}) space of size $2^{\aleph_1}$ and weight $2^{\aleph_0}$, with no isolated points, which does not include a copy of $2^{\omega_1}$ and indeed does not even have a closed subset mapping onto $2^{\omega_1}$.  CH implies the space is $\omega_1$-Rothberger \cite{D}.
\end{ex}

Under CH then, there is a space which is $\omega_1$-Rothberger but also not projectively $\aleph_1$.  Thus ``indestructibility" is the appropriate generalization of ``Rothberger" to use in attempting to generalize Proposition \ref{Prop4}.  Let us make the following definition:

\begin{defn}[{\cite{TU}}]
The \defined{$\mathbf{\aleph_1}$-Borel Conjecture} is the assertion that a Lindel\"of space is indestructible if and only if it is projectively $\aleph_1$.
\end{defn}

There have been several quite different attempts to generalize Borel's Conjecture - see \cite{C}, \cite{GS}, \cite{HS}.

\begin{prop}[\cite{TU}]\label{Prop6}
L\'evy-collapse an inaccessible cardinal to $\aleph_2$.  Then CH and the $\aleph_1$-Borel Conjecture hold.
\end{prop}

In fact (see below), the $\aleph_1$-Borel Conjecture implies CH.  The inaccessible is necessary \cite{D}; see below.

It should be straightforward to generalize Theorem \ref{Thm5} (possibly assuming CH) to obtain something like:
\[\tag{$*$}\text{\emph{$\leq \aleph_1$-L-productive indestructible spaces are projectively $\aleph_1$}.}\label{displayEq}\]
In fact, as we shall see below (Corollary \ref{Cor17}), this is consistently false. 

It is instructive to see what happens when one naively tries to prove \eqref{displayEq} by stepping up the proof of Theorem \ref{Thm5} one cardinal, replacing the Cantor set by a copy of $2^{\omega_1}$.  A crucial step in the proof fails: Example \ref{DExample} is a space of size $2^{\aleph_1}$ without isolated points, which does not include a copy of $2^{\omega_1}$, yet under CH has weight $\aleph_1$.  As an ordered space, this space is hereditarily normal, so by \v{S}apirovskii's mapping theorem (see e.g. \cite{J}) cannot have a closed subspace mapping onto $2^{\omega_1}$.

Given that this attempt to generalize the proof of Theorem \ref{Thm5} in order to obtain \eqref{displayEq} fails, is there another way to get it?  Well, of course the $\aleph_1$-Borel Conjecture trivially implies \eqref{displayEq}, but is that extra hypothesis necessary?  It is:

\begin{prop}[\cite{D}]\label{Thm7}
If $\aleph_2$ is not inaccessible in $L$, there is a compact indestructible space of weight $\aleph_1$ and cardinality greater than $\aleph_1$.
\end{prop}

\begin{cor}\label{Cor17}
If $\aleph_2$ is not inaccessible in $L$, there is an L-productive indestructible space which is not projectively $\aleph_1$.
\end{cor}
\begin{proof}
The example of Proposition \ref{Thm7} is the compact line (which has weight $\aleph_1$) obtained from a Kurepa tree \cite{To}.  Compact spaces are obviously L-productive.  Spaces of weight $\aleph_1$ are embeddable in $[0, 1]^{\omega_1}$.
\end{proof}

It follows that \eqref{displayEq} is equiconsistent with the apparently stronger $\aleph_1$-Borel Conjecture, for if \eqref{displayEq} holds, $\aleph_2$ is inaccessible in $L$ and so we can obtain that Conjecture.

Notice incidentally that one could generalize the proof of Theorem \ref{Thm5} if one knew that perfect subspaces of size $\geq \aleph_2$ of $[0, 1]^{\omega_1}$ included destructible compact subspaces.  The $\aleph_1$-Borel Conjecture assures this.  We can't do better; the Kurepa line of Proposition \ref{Thm7} is compact indestructible, and hence has every compact subspace indestructible.  

One might assume that projectively countable spaces are projectively $\aleph_1$; in fact, this is undecidable!

\begin{ex}\label{Example2}
If there is a Kurepa tree without an Aronszajn subtree (as there is in $L$ \cite{D22}), then there is a Lindel\"of linearly ordered $P$-space ($G_\delta$'s open) of weight $\aleph_1$ and size $> \aleph_1$ \cite{JuhaszWeiss}.  Such a space is obviously not
projectively $\aleph_1$, yet every $P$-space is projectively countable.
\end{ex}

On the other hand,

\begin{thm}
The $\aleph_1$-Borel Conjecture implies that projectively countable Lindel\"of spaces are projectively $\aleph_1$.
\end{thm}
\begin{proof}
Projectively countable Lindel\"of spaces are Rothberger \cite{BCM}, \cite{K}, \cite{T} and hence indestructible. By the $\aleph_1$-Borel
Conjecture, they are then projectively $\aleph_1$.
\end{proof}

\section{The $\aleph_1$-Hurewicz Property}
\indent This section was motivated by the idea that, just as \textit{Borel's Conjecture implies that Rothberger spaces are Hurewicz} \cite{T}, we should be able to prove
\begin{thm}\label{ConjDagger}
The $\aleph_1$-Borel Conjecture implies that indestructible Lindel\"of spaces are $\aleph_1$-Hurewicz.
\end{thm}
\noindent where \textit{$\aleph_1$-Hurewicz} is some natural generalization of the usual Hurewicz property.  We should also be able to generalize the classic theorem that \textit{Hurewicz \v{C}ech-complete spaces are $\sigma$-compact} so as to have \textit{$\aleph_1$-Hurewicz} in the hypothesis and \textit{$\aleph_1$-compact} (the union of $\aleph_1$ compact sets) in the conclusion.  We could then prove

\begin{thm}\label{Conj19}
The $\aleph_1$-Borel Conjecture implies that indestructible Lindel\"of $\aleph_1$-\v{C}ech-complete spaces are $\aleph_1$-compact.
\end{thm}
\noindent where ``$\aleph_1$-\v{C}ech-complete" is a natural generalization defined below of ``\v{C}ech-complete".

There are several equivalent definitions of the Hurewicz property.  See e.g. \cite{aT2}, \cite{aL}.  We will use the following generalization of one such equivalent as our definition of $\aleph_1$-Hurewicz, because it enables us to prove Theorems \ref{ConjDagger} and \ref{Conj19}.

\begin{defn}
A Lindel\"of space is \defined{$\aleph_1$-Hurewicz} if whenever $\set{U_\alpha : \alpha <
\omega_1}$ are open sets in $\beta X$ including $X$, there are closed sets $\set{F_\alpha : \alpha < \omega_1}$ in $\beta X$ such that $X \subseteq \bigcup\set{F_\alpha : \alpha < \omega_1} \subseteq \bigcap\set{U_\alpha: \alpha < \omega_1}$.
\end{defn}

\begin{defn}
A space is $\aleph_1$-compact if it is the union of $\aleph_1$ compact sets.
\end{defn}

\noindent\textbf{Note: } ``$\aleph_1$-compact" used to mean what is now called ``countable extent".  It seems appropriate to repurpose the term.

\begin{defn}
A space $X$ is \defined{$\aleph_1$-\v{C}ech-complete} if there are open covers $\set{\mc{U}_\alpha}_{\alpha < \omega_1}$ of $X$ such that any centered family of closed sets which, for each $\alpha$, contains a closed set included in some member of $\mc{U}_\alpha$ has non-empty intersection.
\end{defn}

\begin{thm}\label{nThm21}
$\aleph_1$-Hurewicz, $\aleph_1$-\v{C}ech-complete spaces are $\aleph_1$-compact.
\end{thm}
\begin{proof}
A routine generalization of the standard proof (see e.g. \cite{E}) that \v{C}ech-complete spaces are $G_\delta$'s in their Stone-\v{C}ech compactifications establishes that an $\aleph_1$-\v{C}ech-complete space is a $G_{\aleph_1}$, i.e., is the intersection of $\aleph_1$ open sets in its Stone-\v{C}ech compactification.  The theorem follows immediately.
\end{proof}

\begin{defn}
A space is \defined{projectively $\aleph_1$-compact} (\defined{projectively $\aleph_1$-Hurewicz}) if its continuous image in $[0, 1]^{\omega_1}$ is always $\aleph_1$-compact ($\aleph_1$-Hurewicz).
\end{defn}

The following standard fact follows, e.g., from Lemma 1.0 in \cite{aL}.

\begin{lem} \label{lel}
Let $\mc{U}$ be an open cover of a regular Lindel\"of space $X$. Then there exists a continuous function $f:X\to \mathbb R$ such that $f^{-1}([-n,n])$ is included in a finite union of elements of $\mc{U}$ for every $n\in\omega$.
\end{lem}

Let us note that in the definition of the $\aleph_1$-Hurewicz property the Stone-\v{C}ech compactification $\beta X$ may be replaced by any other one.

\begin{thm} \label{w_1_pr_hur_pr_hur}
Every Lindel\"of projectively $\aleph_1$-Hurewicz space is $\aleph_1$-Hurewicz.
\end{thm}
\begin{proof}
Let $X$ be a Lindel\"of projectively $\aleph_1$-Hurewicz space and $\{W_\alpha:\alpha<\omega_1\}$ be a collection of open subsets of $\beta X$ including $X$. For every $\alpha$ fix a cover $\U_\alpha$ of $X$ by open subsets of $\beta X$ whose closures are subsets of $W_\alpha$. Set $\mc{U}'_\alpha=\{U\cap X:U\in\mc{U}_\alpha\}$. By Lemma~\ref{lel} for every $\alpha$ there exists a continuous function $f_\alpha:X\to \mathbb R$ such that $f_\alpha^{-1}[-n,n]$ is included in a union of finitely many elements of $\mc{U}'_\alpha$ for all $n$. Now set $f:X\to \mathbb R^{\omega_1}$, $f(x)(\alpha)=f_\alpha(x)$. Since $\mathbb R^{\omega_1}$ is homeomorphic to a $G_{\aleph_1}$-subset of $[0,1]^{\aleph_1}$ and $X$ is projectively $\aleph_1$-Hurewicz, there exists a collection $\mathcal K$ of compact subsets of $\mathbb R^{\omega_1}$ such that $\abs{\mathcal K}\leq\aleph_1$ and $f(X)\subseteq\bigcup\mathcal K$. Therefore $X\subseteq\bigcup_{K\in\mathcal K}f^{-1}(K)$.
It also follows from the above that for every $K\in\mathcal K$ and $\alpha\in\omega_1$ the preimage $f^{-1}(K)$ is included in a finite union of elements of $\mathcal{U}'_\alpha$, and hence its closure in $\beta X$ is included in $\bigcap_{\alpha\in\omega_1}W_\alpha$. Thus
\[X\subseteq\bigcup_{K\in\mathcal K}\op{cl}_{\beta X}f^{-1}(K)\subseteq\bigcap_{\alpha<\omega_1}W_\alpha,\]
which completes our proof.
\end{proof}

\begin{cor}
Lindel\"of projectively $\aleph_1$ spaces are $\aleph_1$-Hurewicz.
\end{cor}

Theorems \ref{ConjDagger} and \ref{Conj19} follow immediately.

The reason we are interested in $\aleph_1$-compactness is because by Lemma \ref{Lem8a} we have:

\begin{lem}[{\cite{BT}}]\label{nLem19}
CH implies productively Lindel\"of $\aleph_1$-compact spaces are powerfully Lindel\"of.
\end{lem}

This result could be used to establish that \textit{CH implies productively Lindel\"of \v{C}ech-complete spaces are powerfully Lindel\"of}, since Lindel\"of \v{C}ech-complete spaces are perfect preimages of separable metrizable spaces, which latter have cardinality $\leq 2^{\aleph_0}$, but we shall not do so because it is known without CH that \textit{Lindel\"of \v{C}ech-complete (indeed $p$-)spaces are powerfully Lindel\"of}.

Incidentally, let us mention:

\begin{thm}
The $\aleph_1$-Borel Conjecture implies Lindel\"of \v{C}ech-complete spaces are $\aleph_1$-compact.
\end{thm}
\begin{proof}
The $\aleph_1$-Borel Conjecture implies CH, since $[0, 1]$ is indestructible and has weight $\leq \aleph_1$.  A Lindel\"of \v{C}ech-complete space is a perfect preimage of a separable metric space, and hence is the union of $\leq 2^{\aleph_0}$ compact sets.
\end{proof}

A straightforward generalization of known results is:

\begin{thm}
$2^{\aleph_1} = \aleph_2$ implies $\aleph_1$-L-productive spaces are projectively $\aleph_1$-compact.
\end{thm}

\begin{proof}
Let $X$ be the continuous image in $[0, 1]^{\omega_1}$ of an $\aleph_1$-L-productive space.  Then $X$ is $\aleph_1$-L-productive.  If $X$ is not $\aleph_1$-compact, $Y = [0, 1]^{\omega_1} - X$ is not the intersection of $\aleph_1$ open sets.  Since $w(Y) \leq \aleph_1$ by assumption there are $\leq \aleph_2$ open sets about $Y$ such that every open set about $Y$ includes one.  We may therefore form a strictly decreasing $\omega_2$-sequence $\set{G_\alpha : \alpha < \omega_2}$ of intersections of $\aleph_1$ open subsets of $[0, 1]^{\omega_1}$ about $Y$.  Pick $z_\alpha \in (G_\alpha - G_{\alpha + 1}) \cap X$.  Take $Z = Y \cup \set{z_\beta : \beta < \omega_2}$ and make each $z_\beta$ isolated.  Then $L(Z) \leq \aleph_1$, but $L(X \times Z) > \aleph_1$, contradiction.
\end{proof}

Clearly projectively $\aleph_1$ implies projectively $\aleph_1$-compact implies projectively $\aleph_1$-Hurewicz.

Recall the space obtained from a Kurepa tree with no Aronszajn subtree (Example \ref{Example2}).  Since $P$-spaces are projectively countable, this is an example of a projectively countable Lindel\"of space which is not projectively $\aleph_1$.  In fact, it is not projectively $\aleph_1$-compact.  To see this, note that its weight is $\aleph_1$, so it is embedded in $[0, 1]^{\omega_1}$.  But compact $P$-spaces are finite.\hfill\qedsymbol

Lindel\"of $P$-spaces are Rothberger and hence indestructible \cite{ST},  but the $\aleph_1$-Borel Conjecture is unavailable, so it is not immediately obvious whether or not this space $Y$ is (projectively) $\aleph_1$-Hurewicz.  It is Hurewicz, since Lindel\"of $P$-spaces are Hurewicz \cite{ST}. 

We could use Theorems \ref{ConjDagger} and \ref{nThm21} and Lemma \ref{nLem19} to prove that indestructible, productively Lindel\"of, $\aleph_1$-\v{C}ech-complete spaces are powerfully Lindel\"of, assuming the $\aleph_1$-Borel Conjecture, but we can do better:

\begin{thm} \label{vzad_pered}
Assume CH. Suppose that $X$ is a regular $\aleph_1$-\vc-complete space which is 
productively Lindel\"of. Then $X$ is powerfully Lindel\"of. 
\end{thm}
\begin{proof}
Let $\langle\U_\alpha:\alpha<\w_1\rangle$ be a sequence of open covers of $X$
witnessing its $\aleph_1$-\vc-completeness. Without loss of generality,
each $\U_\alpha$ is locally finite and countable.
Let us write $\U_\alpha$ in the form $\{U^\alpha_n:n\in\w\}$
and consider the relation 
\[R=\{(r,x)\in \w^{\w_1}\times X: x\in\bigcap_{\alpha\in\aleph_1}\overline{U^\alpha_{r(\alpha)}}\}.\]
\begin{claim} \label{cl101}
The set-valued map $R_\mathrm{l}:X \to \omega^{\omega_1}$ assigning 
to $x\in X$ the set $\{r\in\w^{\aleph_1}: (r,x)\in R\}$
is compact-valued and upper semicontinuous.
\end{claim}
\begin{proof}
It is clear that $R_\mathrm{l}(x)$ is closed in $\omega^{\omega_1}$
for all $x\in X$. Moreover, since every $\U_\alpha$ is locally finite,
we conclude that the set $\{r(\alpha): r\in R_\mathrm{l}(x)\}$ is finite 
for every $x\in X$ and $\alpha\in\omega_1$. Thus $R_\mathrm{l}$ is compact-valued.

Now let $O\subseteq \omega^{\omega_1}$ be an open set including $R_\mathrm{l}(x)$
for some $x\in X$. Passing to a subset of $O$ including $R_\mathrm{l}(x)$,
if needed, we may additionally assume that 
$O=\bigcup\{[s]:s\in \pr_A(R_\mathrm{l}(x))\}$ for some $F\in [\w_1]^{<\w}$,
where $\pr_A:\w^{\omega_1}\to\w^A$ is the natural projection map for every $A\subseteq\w_1$
and $[s]=\{r\in R_\mathrm{l}(X):r\uhr F =s\}$ for all $s\in\w^F$.
Set 
\[U=X\setminus \bigcup\{\bigcap_{\alpha\in F}\overline{U^\alpha_{t(\alpha)}}\: : \: t\in\w^F\setminus  \pr_F(R_\mathrm{l}(x))\}.\]
Then $x\in U$. Moreover, 
since all $\U_\alpha$'s are locally finite, so is the family
$\{\bigcap_{\alpha\in F}\overline{U^\alpha_{t(\alpha)}}\: : \: t \in  \w^F\}$,
and hence $U$ is open. A direct verification shows that
$R_\mathrm{l}(y)\subseteq O$ for all $z\in U$, which completes our proof.
\end{proof}
\begin{claim} \label{cl102}
The set-valued map $R_\mathrm{r}: R_\mathrm{l}(X)=\bigcup_{x\in X}R_l(x) \to X$ assigning 
to $r$ the set $ \bigcap_{\alpha\in\w_1}\overline{U^\alpha_{r(\alpha)}}=\{x\in X:(r,x)\in R\}$
is compact-valued and upper semicontinuous.
\end{claim}
\begin{proof}
Given any $r\in R_\mathrm{l}(X)$ let us observe that $R_\mathrm{r}(r)$ is closed in $X$ and therefore Lindel\"of.
If $R_\mathrm{r}(r)$ is not compact then there exists a decreasing sequence $\langle Z_n:n\in\w\rangle$ of closed subsets 
of $R_\mathrm{r}(r)$ with empty intersection. 
 Set $C_n= \overline{U^n_{r(n)}}\cap Z_n$ for $n\in\w$ 
and $C_\alpha=\overline{U^\alpha_{r(\alpha)}}$ for $\alpha\in\omega_1\setminus \w$.
Then the sequence 
$\langle C_\alpha:\alpha\in\w_1 \rangle$ is centered, $C_\alpha$ is included in the closure of some element of
$\U_\alpha$ for all $\alpha$, and $\bigcap_{\alpha\in\w_1}C_\alpha=\emptyset$,
a contradiction.

Let us fix an open neighborhood  $U$ of $R_\mathrm{r}(r)$.
Then there exists a finite $F\subseteq\w_1$ such that $\bigcap_{\alpha\in F}\overline{U^\alpha_{r(\alpha)}}\subseteq U$,
as otherwise the family $ \{\bigcap_{\alpha\in F}\overline{U^\alpha_{r(\alpha)}}\setminus U : F\in [\w_1]^{<\w}\} $
would be centered and have empty intersection, thus contradicting the choice of the sequence $\langle \U_\alpha:\alpha<\w_1\rangle$.  
It follows from the above that $R_\mathrm{r}([r\uhr F])\subseteq U$, which completes the proof.
\end{proof}
We are now in a position to finish the proof of Theorem~\ref{vzad_pered}.
Since $X$ is productively Lindel\"of and $R_\mathrm{l}$ is compact-valued and upper semicontinuous, 
so is its image $R_\mathrm{l}(X)\subseteq \w^{\omega_1}$. By CH and Lemma \ref{Lem4Prea} all productively Lindel\"of spaces of weight
$\aleph_1$ are powerfully Lindel\"of, and hence so is 
$R_\mathrm{l}(X)$. Since $R_\mathrm{r}$ is compact-valued and upper semicontinuous
and $X=R_\mathrm{r}(R_\mathrm{l}(X))$, we conclude that $X$ is powerfully Lindel\"of as well.
\end{proof}

\section{Projective $\sigma$-compactness in finite powers does not imply productive Lindel\"ofness}
\indent E. A. Michael \cite{M22} proved under CH that productively Lindel\"of spaces are projectively $\sigma$-compact.  In fact, since finite powers of productively Lindel\"of spaces are productively Lindel\"of, under CH they are projectively $\sigma$-compact.  It is natural to wonder whether having finite powers projectively $\sigma$-compact is sufficient, perhaps assuming CH, to conclude productive Lindel\"ofness.  It isn't; an example of Todorcevic \cite{T22} will establish this.

\begin{ex}
There is a $\gamma$-space which has all finite powers projectively countable but is not productively Lindel\"of.
\end{ex}

Recall the definition of a \textit{$\gamma$-space}:

\begin{defn}[\cite{GN22}]
A cover of $X$ is an \defined{$\omega$-cover} if each finite subset of $X$ is included in some member of the cover.  A space is a \defined{$\gamma$-space} if for every open $\omega$-cover $\mc{U}$, there is a sequence of elements of $\mc{U}$, $U_n$, $n < \omega$, such that every member of $X$ is in all but finitely many $U_n$'s.
\end{defn}

$\gamma$-spaces are Lindel\"of; in fact, finite powers of $\gamma$-spaces are $\gamma$ \cite{JMSS22}.  It is easy to see that every continuous image of a $\gamma$-space is a $\gamma$-space.

\begin{lem}
All metrizable $\gamma$-spaces are zero-dimensional.
\end{lem}
\begin{proof}
By II.3.3 and II.3.6 respectively of \cite{Arh}, $\gamma$-spaces are $\phi$-spaces, and $\phi$-spaces have small inductive dimension $0$.  But for Lindel\"of metrizable spaces, that is the same as being $0$-dimensional.
\end{proof}

In \cite{T22}, Todorcevic constructs a stationary Aronszajn line which is $\gamma$, projectively countable, and not productively Lindel\"of.  We claim that his space actually has all finite powers Lindel\"of and projectively countable.  It will suffice to prove the following claim, for then all even powers of $X$ -- and hence all finite powers of $X$ -- are projectively countable.

\begin{claim}\label{easyClaim}
If $X$ is projectively countable, $X^2$ is Lindel\"of, and all continuous metrizable images of $X^2$ are zero-dimensional,
then $X^2$ is projectively countable.
\end{claim}
\begin{proof} 
Let $f:X^2\to Y$ be a continuous map for some metrizable $Y$. We need to show that $f[X^2]$ is countable.  Without loss of generality, $f$ is surjective, and hence $Y$ is zero-dimensional and has countable weight.
Let $\mathcal{B}$ be a countable base of $Y$ consisting of clopen sets and
$\mathcal{C} = \set{f^{-1}(B) : B\in\mathcal B}$. Then every $C\in\mathcal{C}$ is a clopen subset of 
$X^2$, and thus it may be written as a union $\bigcup_{i\in I_C}U_{i} \times V_{i} $ 
for some clopen subsets $U_{i}, V_{i}$ of $X$. Since $X^2$ is Lindel\"of, we can assume that
each $I_C$ is countable. It will be also convenient for us to assume that $I_{C_0}\cap I_{C_1}=\emptyset$
if $C_0\neq C_1$. 
Set $I=\bigcup_{C\in\mathcal C} I_C$ and consider maps $g,h:X\to 2^I$
defined as follows: $g(x)(i)=1$ (respectively $h(x)(i)=1$) if and only if $x\in U_i$ (respectively $x\in V_i$).
It follows from the above that $g$ and $h$ are continuous. Since $I$ is countable,
$g[X]$ and $h[X]$ are countable as well.
We claim that  if $g(x_0)=g(x_1)$ and $h(y_0)=h(y_1)$ then $f(x_0,y_0)=f(x_1,y_1)$.
(This easily implies that $f[X^2]$ is countable.) Suppose that $f(x_0,y_0)\neq f(x_1,y_1)$.
Then there exists $B\in\mathcal B$ such that $f(x_0,y_0)\in B$ but $f(x_1,y_1) \not\in B$.
Then $(x_0,y_0)\in C$ but $(x_1, y_1) \not\in C$, where $C=f^{-1}(B)\in\mathcal C$.
Let $i\in I_C$ be such that $(x_0,y_0)\in U_i\times V_i$ and notice that 
$(x_1,y_1)\not\in U_i\times V_i$. This means that either $x_1\not\in U_i$
or $y_1\not\in V_i$. In the first case we have $g(x_0)(i)\neq g(x_1)(i)$, while in the second case $h(y_0)(i)\neq h(y_1)(i)$. In any case,
$(g(x_0),h(y_0))\neq (g(x_1),h(y_1)) $, which completes our proof.
\end{proof}

We do not know whether Todorcevic's space is powerfully Lindel\"of.  If it is, it would show that even the addition of ``powerfully Lindel\"of" to ``projectively $\sigma$-compact in finite powers" would fail to characterize productive Lindel\"ofness.

\nocite{*}
\bibliographystyle{acm}
\bibliography{Zdomskyy}

{\rm -- Haosui Duanmu, Department of Statistics, University of Toronto, Toronto, Ontario M5S 3G3, CANADA}

{\it e-mail address:} {\rm duanmuhaosui@hotmail.com}
\\ \\
\indent{\rm -- Franklin D. Tall, Department of Mathematics, University of
  Toronto, Toronto, Ontario M5S 2E4, CANADA}

{\it e-mail address:} {\rm f.tall@utoronto.ca}
\\ \\ 
\indent{\rm -- Lyubomyr Zdomskyy, Kurt G\"odel Research Center for Mathematical Logic, W\"aringer Strasse 25, 1070 Wien, AUSTRIA}

{\it e-mail address:} {\rm lyubomyr.zdomskyy@univie.ac.at}

\end{document}